\documentclass[10pt,a4paper]{article}
\usepackage{amsfonts,amsmath,amssymb, times, amsthm}
\usepackage{mhequ}
\usepackage{color}
\usepackage{hyperref}
\newtheorem{theorem}{Theorem}
\newtheorem{lemma}{Lemma}
\newtheorem{proposition}[theorem]{Proposition}

\theoremstyle{definition}

\newtheorem{remark}[lemma]{Remark}
\newtheorem{example}[lemma]{Example}

\numberwithin{equation}{section}
\def\sectionmark#1{\markboth{}{}}

\begin{document}

\def\ssm{\smallsetminus}
\newcommand{\tmname}[1]{\textsc{#1}}
\newcommand{\tmop}[1]{\operatorname{#1}}
\newcommand{\tmsamp}[1]{\textsf{#1}}
\newenvironment{enumerateroman}{\begin{enumerate}[i.]}{\end{enumerate}}
\newenvironment{enumerateromancap}{\begin{enumerate}[I.]}{\end{enumerate}}

\newcounter{problemnr}
\setcounter{problemnr}{0}
\newenvironment{problem}{\medskip
  \refstepcounter{problemnr}\small{\bf\noindent Problem~\arabic{problemnr}\ }}{\normalsize}
\newenvironment{enumeratealphacap}{\begin{enumerate}[A.]}{\end{enumerate}}
\newcommand{\tmmathbf}[1]{\boldsymbol{#1}}

\def\paral{/\kern-0.55ex/}
\def\parals_#1{/\kern-0.55ex/_{\!#1}}
\def\bparals_#1{\breve{/\kern-0.55ex/_{\!#1}}}
\def\n#1{|\kern-0.24em|\kern-0.24em|#1|\kern-0.24em|\kern-0.24em|}

\newcommand{\A}{{\bf \mathcal A}}
\newcommand{\B}{{\bf \mathcal B}}
\def\C{\mathbb C}
\newcommand{\D}{{\rm I \! D}}
\newcommand{\dom}{{\mathcal D}om}
\newcommand{\pathR}{{\mathcal{\rm I\!R}}}
\newcommand{\Nabla}{{\bf \nabla}}
\newcommand{\E}{{\mathbf E}}
\newcommand{\Epsilon}{{\mathcal E}}
\newcommand{\F}{{\mathcal F}}
\newcommand{\G}{{\mathcal G}}
\def\g{{\mathfrak g}}
\newcommand{\HH}{{\mathcal H}}
\def\h{{\mathfrak h}}
\def\k{{\mathfrak k}}
\newcommand{\I}{{\mathcal I}}
\def\LL{{\mathbb L}}
\def\law{\mathop{\mathrm{ Law}}}
\def\m{{\mathfrak m}}
\newcommand{\K}{{\mathcal K}}
\newcommand{\p}{{\mathfrak p}}
\def\P{\mathbb{P}}
\newcommand{\R}{{\mathbb R}}
\newcommand{\Rc}{{\mathcal R}}
\def\T{{\mathcal T}}
\def\M{{\mathcal M}}
\def\N{{\mathcal N}}
\newcommand{\pnabla}{{\nabla\!\!\!\!\!\!\nabla}}
\def\X{{\mathbb X}}
\def\Y{{\mathbb Y}}
\def\L{{\mathcal L}}
\def\1{{\mathbf 1}}
\def\half{{ \frac{1}{2} }}
\def\vol{{\mathop {\rm vol}}}
\def\euc{{\mathop {\rm eul}}}

\newcommand{\term}{{1}}
\newcommand{\tm}{{t}}
\newcommand{\stm}{{s}}
\newcommand{\pole}{{y_0}}

\def\ad{{\mathop {\rm ad}}}
\def\Conj{{\mathop {\rm Ad}}}
\def\Ad{{\mathop {\rm Ad}}}
\newcommand{\const}{\rm {const.}}
\newcommand{\eg}{\textit{e.g. }}
\newcommand{\as}{\textit{a.s. }}
\newcommand{\ie}{\textit{i.e. }}
\def\s.t.{\mathop {\rm s.t.}}
\def\esssup{\mathop{\rm ess\; sup}}
\def\Ric{{\mathop{\rm {Ric}}}}
\def\ric{{\mathop{\rm ric}}}
\def\div{\mathop{\rm div}}
\def\ker{\mathop{\rm ker}}
\def\Hess{\mathop{\rm Hess}}
\def\Image{\mathop{\rm Image}}
\def\Dom{\mathop{\rm Dom}}
\def\id{\mathop {\rm id}}
\def\Image{\mathop{\rm Image}}
\def\Cyl{\mathop {\rm Cyl}}
\def\Conj{\mathop {\rm Conj}}
\def\Span{\mathop {\rm Span}}
\def\trace{\mathop{\rm trace}}
\def\ev{\mathop {\rm ev}}
\def\supp{{\mathrm supp}}
\def\Ent{\mathop {\rm Ent}}
\def\tr{\mathop {\rm tr}}
\def\graph{\mathop {\rm graph}}
\def\loc{\mathop{\rm loc}}
\def\so{{\mathfrak {so}}}
\def\su{{\mathfrak {su}}}
\def\u{{\mathfrak {u}}}
\def\o{{\mathfrak {o}}}
\def\pp{{\mathfrak p}}
\def\gl{{\mathfrak gl}}
\def\hol{{\mathfrak hol}}
\def\z{{\mathfrak z}}
\def\t{{\mathfrak t}}
\def\<{\langle}
\def\>{\rangle}
\def\span{{\mathop{\rm span}}}
\def\diam{\mathrm {diam}}
\def\inj{\mathrm {inj}}
\def\Lip{\mathrm {Lip}}
\def\Iso{\mathrm {Iso}}
\def\Osc{\mathop{\rm Osc}}
\renewcommand{\thefootnote}{}
\def\supp{\mathrm {Supp}}
\def\V{\mathbb V}
\def\vol{{\mathop {\rm vol}}}
\def\cut{{\mathop {\rm Cut}}}
\def\Lip{{\mathop {\rm Lip}}}
\def\Cyl{{\mathop {\rm Cyl}}}

\def\paral{/\kern-0.55ex/}
\def\parals_#1{/\kern-0.55ex/_{\!#1}}
\def\bparals_#1{\breve{/\kern-0.55ex/_{\!#1}}}
\def\n#1{|\kern-0.24em|\kern-0.24em|#1|\kern-0.24em|\kern-0.24em|}
\def\f{\frac}
\title{Generalised Brownian bridges: examples}
\author{Xue-Mei Li \\ {\small Department of Mathematics,  Imperial College London \footnote{Email: xue-mei.li@imperial.ac.uk}}}
\date{Version accepted by Markov Processes and Related Fields}
\maketitle

\begin{abstract}

We observe that the probability distribution of the Brownian motion with drift  $-c \f x {1-t}$ where $c\not =1$ is singular with respect to that of the classical Brownian bridge  measure on $[0,1]$, while their Cameron-Martin spaces are equal set-wise if and only if $c> \f 12$, providing also examples of  exponential martingales on $[0,1)$ not extendable to a continuous martingale on $[0,1]$. Other examples of generalised Brownian bridges are also studied.
 \end{abstract}

{\bf Key words.}  Brownian bridges, time-dependent singular drifts, Gaussian measures, equivalence, Cameron-martin Spaces.

\footnote{AMS Mathematics Subject Classification :  60Dxx, 60 H07,  58J65, 60Bxx} 

\section{Introduction}
For an $L^2$ analysis on the loop space over a manifold  pinned at $x_0$ and $y_0$, it is standard to use the  Brownian  bridge measure, the latter is a Brownian motion from $x_0$ conditioned to reach $y_0$ at $1$ and is also given by the solution to the stochastic differential equation (SDE)
$dz_t=\circ dx_t+\nabla \log p_{1-t}(x_0, z_t)dt$ with initial value $x_0$, where $x_t$ is a Brownian motion and $p_t(x,y)$ denotes the heat kernel. This measure is notoriously difficult  to understand for it would  involve precise analysis of  the gradient and the Hessian of the logarithm of the heat kernel.  For example the class of manifolds for which  the Poincar\'e inequality are known to hold for its Brownian bridge measure are limited: they are
$\R^n$ \cite{Gross}, the hyperbolic space \cite{CLW}, and a class of asymptotically flat manifolds \cite{Aida15}.
On the other hand an integration by parts formula was shown to hold on a manifold with a pole for 
the probability measure induced by the semi-classical Brownian bridge  \cite{semi-classical}. The latter solves an SDE with a time-dependent gradient drift  which differs from $\nabla \log k_{1-t}(x_0,\cdot)$, in general,
but appears to be easier to treat. Hence it is interesting to know whether the two measures are equivalent \cite{semi-classical}.  We also note that the heat kernels measure on the loop space over a simply connected compact Lie group and  the Brownian bridge measure are proven to be equivalent \cite{Aida-Driver}.

 The equivalence of two measures on the loop space are subtle.
 The purpose of this article is to give simple examples of generalised bridge measures, which we introduce shortly,  that are not equivalent. 
 The first class of examples are  the probability measure induced by the solution of the stochastic differential equation  $dz_t=dB_t-\f c{1-t} z_t\;dt$, where $(B_t)$ is a  Brownian motion on $\R^n$. They  induce a family of Gaussian measures, $\nu^{(c)}$,  on $C_{0,0}([0,1]; \R^n)$, the loop space of continuous paths  
 from $[0,1]$ to $\R^n$ pinned at $0$. Gaussian measures are quasi-invariant under translation by a vector from its Cameron-Martin space and they are determined by their Cameron-Martin space and their covariance operators. The Cameron-Martin space for the Wiener measure is $H_0$, the space of finite energy, and that for the Brownian bridge measure $ \nu^{(1)}$ is its sub-space $H_{0,0}=\{ h: [0,1]\to \R^n: h\in H_0, h(1)=0\}$.

We show that  $\nu^{(c)}$ is singular with respect to $\nu^{(1)}$ unless $c=1$, while their  Cameron-Martin spaces are the same, as sets, for all $c> \f 12$. 
We also note that the Cameron-Martin space of the Gaussian measure given by the SDE $dz_t=dB_t-\f {z_t} {(1-t)^\alpha} dt$, where $\alpha>1$,  is not the same as $H_{0,0}$. Finally we give examples of generalised `Brownian bridge measures' which are equivalent to $ \nu^{(1)}$.

The Cameron-Martin \cite{Cameron-Martin} theorem states that  the Wiener measure on the Wiener space $C_0([0,1], \R)$ is quasi-invariant under the linear transformation $x\mapsto x+h$ if and only if $h$ is a Cameron-Martin vector, i.e. $h$ belongs to the Sobolev space $H_0$. By quasi-invariance we mean that the pushed forward Wiener measure is equivalent to the Wiener measure (i.e.  their null sets are preserved).
This is related to the popular Girsanov transform for martingales, for Brownian motions and for stochastic differential equations.  Following this, Woodward gave a sufficient condition on $L$ to ensure
the linear transformation on $C_0([0,1];\R)$,
$x(\cdot) \mapsto x(\cdot)+\int_0^1 L(\cdot,s) dx(s)$,  to be one to one and onto, where the integral is the Wiener integral.
In \cite{Shepp}, Shepp gave a necessary and sufficient condition for  the stochastic process
$x(t)= (\dot h(t))^{-\f 12} W\big(h(t)\big)$, where $h$ is an increasing function with $h(0)=0$,  to be equivalent to the Wiener process $W(t)$.
This generalises a result of Segal \cite{Segal} and others on non-stochastic transformations and uses a representation of Hitsuda \cite{Hitsuda}. See also Varberg \cite{Varberg}. For further discussions on Gaussian measures see \cite{Bogachev}.

Coming back to  Brownian bridges,  there are other recent related studies.  Lupu, Pitman and Tang  \cite{Lupu-Pitman-Tang, Pitman-Tang} studied the probability distributions 
of Brownian bridges under Vervaat transformation, using the first time at which the minimum of the Brownian motion is attained, in terms of path decompositions  and Brownian excursions. They were inspired by the study of the quartile functions, see also \cite{Embrechts-Rogers-Yor}. 
One of their main questions is a Skorohod embedding type problem for the Brownian bridge: is  
${W(\tau+\cdot)-W(\tau)}$ a Brownian bridge for some random time $\tau$?  This problem belongs also to the domain of
shift coupling, see \cite{Aldous-Thorisson}. See also \cite{Bertoin-Chaumont-Pitman,Biane-Legall-Yor, Vigon, Vervaat, Mazzolo, Barczy-Kern}.% {Matsumoto-Taniguchi}.

\section{Generalised Brownian bridges and examples}

We define a {\it generalised Brownian bridge process}, say $(z_t, t\in [0,1])$,
between $x_0$ and $y_0$ with terminal time $1$, to be a stochastic process satisfying the following conditions:
\begin{itemize}
\item [(1)] $\{z_t, t\in [0, 1)\}$ is a Markov process with infinitesimal generator of the form: $ \f 12 \Delta- f(t) \nabla (\f {r^2(\cdot, y_0)} {2})$ where $f$ is a suitably smooth real valued  function on $[0,1)$
with $\lim_{t\to 1}f(t)=\infty$ and $r$ is a distance function on the state space;
\item [(2)] $ \lim_{t\to 1} z_t~=~y_0$ a.e..
\end{itemize}
Their probability distributions on the space of continuous paths are called generalised Brownian bridge measures.
See also \cite{Li-hypoelliptic-bridge, semi-classical} for other types of Brownian bridges.
In the rest of the section we take the space to be $\R^n$ and the end points to be $x_0=y_0=0$.

A function $G: [0,1)\times [0,1)\to \R_+$ is said to be an approximation to the identity~if the following holds:
\begin{enumerate}
\item [(1)] $\lim_{t\uparrow 1} \int_0^t G(s,t) ds=1$,
\item [(2)]  $\lim_{t\uparrow 1} \int_0^{t_0} G(s,t) ds=0$ for any $t_0<1$.
\end{enumerate}
If $G$ is an approximation to the identity and if  $\sigma: [0,1]\to \R$ is a  continuous function,
then $\lim_{t\to 1}\left( \int_0^t G(s,t) \sigma(s) ds-\sigma(t) \right)=0$.
% Indeed, this follows from
%$$ \int_0^t G(s,t) [\sigma(s) -\sigma(t)]ds=\left(\int_0^{t_0}+\int_{t_0}^t\right) G(s,t) [ \sigma(s) -\sigma(t)]\,ds.$$

\begin{proposition}
Let $f: [0,1)\to \R$ be a function such that for any $t_0\in (0, 1)$,
$$\lim_{t\uparrow 1} \int_0^t f(s)\, ds=\infty, \quad  \int_0^{t_0}f(s) \, ds<\infty.$$
Then $G_f(s,t):= f(s) \, e^{-\int_s^t f(r)dr}$ is an approximation of the identity. It follows that  the solution to the  SDE  $dy_t=dB_t-f(t)\,y_t \,dt$ on $\R^n$, where $  t<1$ and $y_0=0$,
is a generalised Brownian bridge from $0$ to $0$. 
\end{proposition}
\begin{proof}
Since $\int_0^{t_0} G_f(s,t)ds=e^{-\int_0^t f(s)ds} \left(e^{\int_0^{t_0} f(s)ds}-1\right)$,
and  by the assumption that $ \lim_{t\uparrow 1}e^{-\int_0^t f(s)ds}=0$, $G_f$ is indeed an approximation to the identity.
The solution to the SDE is explicit and given by the formula $y(t)=e^{-\int_0^t f(r)dr}\int_0^t e^{\int_0^s f(r)dr}dB_s$. An integration by part shows that,
$$y(t)=B_t- e^{-\int_0^t f(r)dr}\int_0^t  B_s\f{d } {ds}  e^{\int_0^s f(r)dr}\, ds\to 0
=B_t -\int_0^t G_f(s,t) B_s\, ds,$$
proving the proposition.
 \end{proof}

We will need the following estimates.

\begin{lemma}\label{lemma1}
If  $h\in H$ then  $\f{ h(1)-h(u)}{1-u}\in L^2$. If $f\in L^2([0,1],\R^n)$ and  $c\in (\f 12, 1]$,  set
$$g(x)=\f 1 {(1-x)^{1-c}}\int_0^x \f {f(y)}{(1-y)^c} \,dy.$$
Then $\|g\|_{L^2} \le \f 2 {2c-1}  \|f\|_{L^2}$.
\end{lemma}
\begin{proof}
If  $h\in H$, it is well known that  $\f{ h(1)-h(u)}{1-u}\in L^2$. In fact
$$\begin{aligned}
\int_0^t  \f {|h(s) -h(1)|^2} {(1-s)^2}ds
&=   \f {|h(t)-h(1)|^2} {1-t}-\int_0^t  \f {\<h(s)-h(1), 2\dot h(s)\>} { 1-s}ds\\
&\le \f { |h(t)-h(1)|^2} {1-t}+2 \int_0^t |\dot h(s)|^2 ds+\f 12 \int_0^t \f {| h(s)-h(1)|^2 }{(1-s)^2}ds.
\end{aligned}$$
Let $f\in L^2$, then
$$\begin{aligned}
&(\|g\|_{L^2})^2=\int_0^1 \f {1}{(1-x)^{2-2c}} \int_0^x\int_0^x \f {\<f(y),f(z)\>}{(1-y)^c(1-z)^c} \,dy\,dz\, dx\\
&= \int_0^1 \int_0^1 \f{\<f(y),f(z)\> }{(1-y)^c(1-z)^c} \int_{y\vee z}^1  \f {dx}{(1-x)^{2-2c}}  \,dy\,dz \\
&=\f 1 {2c-1}  \int_0^1 \int_0^1 \f{\<f(y),f(z)\> }{(1-y)^c(1-z)^c}  { (1-y\vee z)^{2c-1} }  \,dy\,dz\\
&=\f 2 {2c-1}  \int_0^1 \int_0^z \f{\<f(y),f(z)\> }{(1-y)^c(1-z)^c}  { (1-z)^{2c-1} }  \,dy\,dz\\
&=\f 2 {2c-1}  \int_0^1\< g(z), f(z)\>dz.
\end{aligned}$$
The required estimate follows from the Cauchy-Schwartz inequality. \end{proof}

For a  real valued  function $f$ on $[0,1)$ we define
$$\Phi(f)(t)=e^{\int_0^t f(s)ds}, \quad t<1.$$

\begin{example}\label{example1}
For $\alpha >1$ set $f_\alpha(t)=\f 1 {(1-t)^\alpha}$ and let $(y_t^\alpha)$ be the solution to
the following SDE on $\R^n$:
$dy_t=dB_t-f_\alpha(t) \,y_t \,dt$ where $t<1$, and $y_0=0$.
 Then $y_t^\alpha$ is a generalised Brownian bridge. Its probability distribution is
singular with respect to the Brownian bridge measure on the loop space~$C_{0,0}\R^n$.
They have different Cameron-Martin spaces.
\end{example}

\begin{proof}
Since $(y_t^\alpha, t\le 1)$ is a Gaussian process, by the Feldman-Hajek theorem it is either equivalent to
the Brownian bridge measure or singular to it. For two Gaussian measure to be equivalent it is a necessary condition that their Cameron-Martin spaces agree. 
The Cameron-Martin space for $(y_t^\alpha)$
is
$$H^{(\alpha)}
=\left\{ k:[0,1]\to \R^n: k(t)=\Phi(-f_\alpha)(t)\int_0^t \Phi(f_\alpha)(s) \f {d}{ds} h(s) ds, h \in H\right\}.$$
This follows from the following fact. Let $\mu$ be a measure on a Banach space $E$ and let $T$ be a linear map from 
$E$ to another Banach space $\hat E$. We denote by $\nu=T_*\mu$  the pushed forward measure. Then the Cameron-Martin space of $\nu$ is the image of the Cameron-Martin space of $\mu$ by $T$. Take $\mu$ to be the Wiener measure, $E=L^2([0,1];\R^n)$ and $T$ the map
$$T(\sigma)(t)=\Phi(-f_\alpha)(t)\int_0^t \Phi(f_\alpha)(s)\; d\sigma(s).$$
After an integration by parts, we see that $T(\sigma)(t)=\sigma(t) -\tilde T(\sigma)(t)$ where  
$$\tilde T(\sigma)(t)= \Phi(-f_\alpha)(t)\int_0^t \f d {ds}\left[\Phi(f_\alpha)(s) \right]\sigma(s)ds.$$
Since $ \Phi(-f_\alpha(t))\to 0$ as $t\to 1$, ${\Phi(-f_\alpha)(t)}\f d{ds}[{\Phi(f_\alpha)(s) }]$ is
 an approximation of the identity and  $\lim_{t\to 1} y_t^\alpha=0$ for $\alpha \ge 1$.

For $\alpha=1$ it is easy to see that $H^{(1)}=H_{0,0}$, where  $H_{0,0}=\{h \in H: h(1)=0\}$. Suppose that  $k$ and $ h$ are related by the formula
$k(t)=(1-t) \int_0^t \f {\dot h(s) }{1-s}\,ds$. If  $k \in H_{0,0}$ then
$\dot h(s)=\dot k(s)+\f {k(s)} {1-s}\in L^2([0,1])$; if $h\in H$ then
$\dot k(t)=\dot h(t)-\int_0^t \f {\dot h(s)} {1-s} ds$  belong to $ L^2([0,1])$. Both by Lemma \ref{lemma1}.
For $\alpha>1$ consider  the inverse map $T^{-1}$:
$$ h=T^{-1}(k)=\Phi(-f_\alpha) \f d {dt}\left[ \Phi(f_\alpha)k\right].$$
For $k \in H_{0,0}$, 
$\dot h(t)=\dot k(t)-   \Phi(-f_\alpha)\; \f d {dt} [\Phi(f_\alpha) ] \; k(t)$
 belongs to $L^2$ if and only if the second term, $ \Phi(-f_\alpha)\f d {dt}[ \Phi(f_\alpha) ] k(t)=\f {k(t)} {(1-t)^\alpha}$, does. It is possible to find
$k\in H_{0,0}$ such that $\f {k(t)} {(1-t)^\alpha}$ does not belong to $L^2$, e.g. take $k(t)$ to be of the order $(1-t)^{\f 12+\epsilon}$ where $\epsilon<\alpha -1$. This means that $H^{(\alpha)}\not =H$, and
the generalised Brownian bridge measure is not equivalent to the Brownian bridge measure.
\end{proof}

\begin{remark}
Let $(M_t^\alpha, t<1)$ denote the exponential martingale, in the Girsanov transform from the Brownian bridge to the generalised Brownian bridge.  It cannot be extended to a martingale on $[0,1]$, see Theorem 3.1 of \cite{Revuz-Yor}. 
 In view of  the use of strict local martingales in the study for financial bubbles, see \cite{Delbaen-Schachermayer,  Mijatovic-Urusov, Pal-Protter, Protter15, ELY97, ELY99, Li-slm}, this type of exponential martingales might be interesting in mathematical finance. The same can be said of
the class of, somewhat surprising, examples below.
\end{remark} 
%It has a number of `square roots', one of which is $Q=SS^*$ where
%$$Sf(t)=(1-t)\int_0^t \f {f(s)}{1-s} ds, \qquad S^*f(t)=\int_t^1 \f {1-s}{1-t }f(s) ds.$$
\begin{example}
Let  $c>0$ be a real number and let $(z^{(c)}(t))$ be the solution to the SDE
$dz_t=dB_t-\f c{1-t} z_t\,dt$ with $z_0=0$. The solution will be called the generalised Brownian bridge with parameter $c$.
Then \begin{itemize}
\item [(a)] $(z^{(c)}_t)$ is a generalised Brownian bridge.  
\item [(b)] Denote by $ \nu^{(c)}$  its
probability distribution on the loop space.
Then its Cameron-Martin space $H^{(c)}$ agrees with $H^{(1)}$ as a set if and only if $c> \f 12$. 
\item [(c)]  The generalised Brownian bridge measures  $\nu^{(1)} $ and $ \nu^{(c)}$   are mutually singular unless $c=1$. 
\end{itemize}
\end{example}
\begin{proof}
Following  the proof of Example \ref{example1}, we define:
$$T(\sigma)(t)=(1-t)^c\int_0^t \f 1{(1-s)^c} d \sigma(s)=\sigma(t)-c(1-t)^c \int_0^t \f {\sigma(s)}{ (1-s)^{c+1}} ds,$$
the first integral being a stochastic integral. The Cameron-Martin space of the Gaussian distribution of $(z_t^{(c)})$ is
$$H^{(c)}
=\left\{ k:[0,1]\to \R^n: k(t)=(1-t)^c\int_0^t \f{  \dot h(s) }{(1-s)^c}ds, h \in H\right\}.$$
If $k\in H_{0,0}\equiv H^{(1)}$,  let $ h(t):=T^{-1}(k)(t)=(1-t)^c \f d {dt}\left[ (1-t)^{-c}k(t)\right]$.
Since $k(1)=0$,  $\f {k(t)}{1-t} \in L^2$ by Lemma \ref{lemma1}, consequently, $\dot h(t)=\dot k(t) + c \f{ k(t)}{1-t}$ belongs to  $L^2$.  
Hence $T^{-1}$ maps  $H_{0,0}$ to  $H^{(c)}$.

Let $h\in H$. It is clear that $T(h)(0)=0$ and $T(h)(1)=\lim_{t\to 1} (1-t)^c\int_0^t \f{  \dot h(s) }{(1-s)^c}ds=0$. 
 We only need to prove the second term, in the following formula 
 $$\f d{dt} T(h)(t)=\dot h(t)-c(1-t)^{c-1} \int_0^t \f {\dot h(s)}{(1-s)^c} ds,$$
belongs to $L^2$.  For $c\in (\f 12, 1]$ this follows from  Lemma \ref{lemma1}. Let $c>1$.
We observe that  $\int_{u\vee v}^1(1-t)^{2c-2} dt= \f 1{ 2c-1} (1-u\vee v)^{2c-1}$,
and $(1-t)^{2c-1}$ makes sense on $[0,1]$, so  we have the estimate below.
$$\begin{aligned}
&\int_0^1 (1-t)^{2c-2} \left|\int_0^t \f {\dot h(s)}{(1-s)^c} ds\right|^2 dt
= \int_0^1 \int_0^1 \f{ \<\dot h(u), \dot h(v)\> }{(1-u)^c(1-v)^c}  \f{ (1-u\vee v)^{2c-1} }{2c-1} dudv\\
&=\f 2 {2c-1} \int_0^1 \int_0^u \left\< \f{ \dot h(v) }{1-v},  \dot h(u) \f{(1-u)^{c-1}}{(1-v)^{c-1}} \right\>  dv\,  du\le  \f 2 {2c-1} \int_0^1 \left| \int_0^u\f{ \dot h(v) }{1-v} dv\right|
|\dot h(u) |  du,
\end{aligned}$$
which is finite by Cauchy-Schwartz inequality and by Lemma \ref{lemma1}. Hence $H^{(c)}$ is contained in $H_{0,0}$ for $c>\f 12$.

If $c\le \f 12$,  for
$\int_0^1 (1-t)^{2c-2} \left|\int_0^t \f {\dot h(s)}{(1-s)^c} ds\right|^2 dt$ to be finite  it is necessary that
$\lim_{t\to 1} \int_0^t \f {\dot h(s)}{(1-s)^c} ds=0$, a condition not satisfied by every $h\in H$.
We have proved that $H^{(c)}=H_{0,0}$ for and only for $c> \f 12$. In particular the measure of
 the generalised bridge process with parameter $c\le \f 12$ is singular with respect to the bridge measure.

 We need a more subtle argument to show that he generalised bridge process with parameter $c>\f 12$ is singular 
with respect to the bridge measure.  Since  their Cameron-Martin spaces agree  as a set
we will need to study their covariance operators. 
Set
$$(R \phi)(t)=\int_0^1 (s\wedge t -st) \phi(s) ds, \quad Q_c(s,t)= \int_0^{s\wedge t} \f {(1-t)^c   (1-s)^c} {(1-r)^{2c}}dr.$$
Let  $R_c: L^2([0,1];\R^n) \to L^2([0,1];\R^n) $ denote the covariance operator of the distribution of $(z_t^{(c)})$:
$$\begin{aligned}(R_c \phi)(t)&= \int_0^1  \phi(s) Q_c(s,t)\;ds
=\f {(1-t)^c } {2c-1} \int_0^1  \phi(s) (1-s)^c \left[(1-{s\wedge t})^{1-2c}-1\right] \;ds.
\end{aligned}$$
In particular  $R$ is  the covariance operator for the Brownian bridge.
By the Feldman-Hajek theorem the two probability measures are equivalent if
and only if  $R^{-\f 12} R_c R^{-\f 12} $ is the sum of an identity map and a Hilbert-Schmidt operator, where $R^{1/2}$ is the square root defined by functional  calculus \cite{DaPrato-Zabczyk}.

We define an operators $A: L^2([0,1];\R^n)\to L^2([0,1];\R^n)$ by the formula:
$$(Af)(t)=\int_0^t f(s)ds-t \int_0^1 f(s)ds.$$
We also define $A^*: L^2([0,1];\R^n)\to L^2([0,1];\R^n)$ by the formula
$$(A^* \phi)(t)=\int_t^1 \phi(s) ds- \int_0^1 s\phi(s)ds.$$
The image of $A$ is the set of $L^2$ functions with the boundary condition $\phi(0)=\phi(1)=0$
while that of $A^*$ is $E=\{ f\in L^2([0,1];\R^n): \int_0^1 f(s)ds=0\}$. Denote by $(A^*)^{-1}$ the left inverse of 
$A^*$, then $(A^*)^{-1}=-\f d {dt}$, without any boundary condition. Furthermore
 $A \f d{dt} f =f$ when restricted to $H_{0,0}$.
We observe that the image of
$R_c$ is contained in the Cameron-Martin space $H^{(c)}=H_{0,0}$ (it is also dense there).
Restricted to $H_{0,0}$,  the right inverse of $A$, which we denote by $A^{-1}$,  is $\f d{dt}$.

It is easy to see that  $R\f d{dt} \phi=-A \phi$, $-\f d {dt} R \f d{dt} \phi =\phi-\int_0^1 \phi(s) ds$ and
$R=AA^*$. Observe that  $R$ is invertible on $E$, $A$ and $R^{1/2}$ differ by a unitary operator.
To check whether $R^{-\f 12} R_c R^{-\f 12} $ is the sum of an identity map and a Hilbert-Schmidt operator,
it is sufficient to prove that $A^{-1} R_c (A^*)^{-1} $ is the sum of an identity map and a Hilbert-Schmidt operator. Indeed, if $A^{-1} R_c (A^*)^{-1} =I+K$, where $K$ is a Hilbert-Schmidt operator, then
$R^{-\f 12} R_c R^{-\f 12}=I +(R^{-\f 12}A)K(R^{-\f 12}A)^*$.

Let $\phi: [0,1]\to \R^n$ be a test function, then
$$\begin{aligned}
&-(R_c (A^*)^{-1}  \phi)(t)
=(R_c \dot \phi )(t)
=\int_0^1 \int_0^{s\wedge t} \f{(1-s)^c(1-t)^c}{(1-r)^{2c}} dr\; \dot \phi(s)\; ds\\
&=(1-t)^c\left[ \int_0^t \int_0^s  \f{dr}{(1-r)^{2c}} (1-s)^c\; \dot \phi(s)\; ds
+  \int_t^1(1-s)^c \dot \phi(s) ds\int_0^t \f {dr} {(1-r)^{2c}}\right].
\end{aligned}$$
 We then apply the operator $A^{-1}$:
 $$\begin{aligned}
\f d {dt}(R_c (A^*)^{-1}  \phi)(t)
&=-\left[\f 1{ (1-t)^{c}}\int_t^1 (1-s)^c \dot \phi(s) \;ds -\f c {1-t} \; (R_c \dot \phi) (t)\right]\\
&=\phi(t)-c\int_t^1 \f {(1-s)^{c-1}} {(1-t)^c} \phi(s) ds
+\f c {1-t} \; (R_c \dot \phi) (t).\end{aligned}$$
We compute the last term,
$$\begin{aligned}
\f c {1-t} \; (R_c \dot \phi) (t)
= &-c(1-t)^{c-1} \left[  \int_0^t \f {\phi(s)} {(1-s)^c}  \; ds
\right]\\
&+c^2\int_0^1\int_0^{s\wedge t}\f {(1-t)^{c-1}(1-s)^{c-1} }{(1-r)^{2c}} dr\; \phi(s) ds.
\end{aligned}$$
Thus,  $$\left(A^{-1}R_c (A^*)^{-1}  \phi\right)(t)=\f d {dt}(R_c (A^*)^{-1}  \phi)(t) =\phi(t)+\int_0^1 q_c(s,t) \phi(s) ds,$$ where
$$\begin{aligned}q_c(s,t)=-c\f{(1-s\vee t)^{c-1}}{(1-s\wedge t)^c}+c^2 {(1-t)^{c-1} (1-s)^{c-1}} \int_0^{s\wedge t} \f {dr}{(1-r)^{2c}}.\end{aligned}
$$
For $c=1$,  $q_1(s,t)=-1$ and $\int_0^1 q_1(s,t) \phi(s)ds$ vanishes on $E$. For $c\not =\f 12$,
$$q_c(s,t)=\f {c(1-c)}{2c-1} \f{(1-s\vee t)^{c-1}}{(1-s\wedge t)^c}- \f{c^2}{1-2c}  (1-t)^{c-1}(1-s)^{c-1}.$$
This is not an $L^2$ function: the second term on the right hand side is $L^2$ integrable for $c>\f 12$ while 
the first term  has a logarithmic singularity unless $c=1$.  
We have proved that the generalised Brownian bridge bridge process  with parameter $c$ and the classical Brownian bridge are mutually singular for any $c$ in $ (0, 1)\cup (1, \infty)$.
 \end{proof}
 
 Finally we observe a perturbative result.

\begin{example}
Let $\delta\in (0, 1/2)$ and $c$ be constants satisfying $c^2\le \f 1 4 (1-2\delta)$.
We consider the equation $dx_t=dB_t -\f {x_t} {1-t} dt+ \f {f(t, x_t)}{(1-t)^\delta}\,dt$ where $f:[0,1]\times \R^n\to \R^n$
satisfies the bound $|f(t, x)|^2\le c|x|^2+c$. Suppose that the SDE is well posed and is conservative.
Then the probability distribution  of $(x_t)$, which we denote by $\nu_2$,  is equivalent to $\nu^{(1)}$.
\end{example}
\begin{proof} 
Let $z_t$ solves $dz_t= dB_t -\f {z_t} {1-t} dt$.  
Since $B_t -tB_1$ is a representation of the Brownian bridge,  we see that
$\E \exp ({a\sup_{0\le t \le 1} |z_t|^2})$ is finite if $4a<1/2$. This follows from Fernique's theorem. 
 (Integrability of exponential of Bessel bridges were studied in \cite{Ikeda-Matsumoto, Pitman-Yor-bridge}.)
For~$t<1$,
$$\f {d  \nu_2}  {d\nu^{(1)}}  =  \exp { \left(   \int_0^t   \left \<  dB_s,  \f {f(s, z_s)} {(1-s)^\delta}  \right\>
-\f 12  \int_0^t    \f  { \left|  f(s, z_s) \right|^2}  {(1-s)^{2\delta}}   ds   \right)}.$$
By the assumption,  $$ \f 12 \int_0^t    \f  { \left|  f(s, z_s) \right|^2}  {(1-s)^{2\delta}}   ds  \le (1+\sup_{s\in [0,1)} |z_s|^2) \f {c^2} {2(1-2\delta)}.$$
which is exponentially integrable. Let $N_t= \int_0^t   \left \<  dB_s,  \f {f(s, z_s)} {(1-s)^\delta}  \right\>$.
We invoke the Novikov criterion to conclude that the exponential martingale of $N_t$  is uniformly  integrable on $[0,1)$ and  converges in $L^1$ as $t$ approaches $1$. It follows that $\nu_2$ is absolutely continuous with respect to $\nu^{(1)}$.
Since $(N_t-\f 12 \<N, N\>_t, t\in [0, 1))$ is $L^2$ bounded, it converges in $L^2$ and so has a finite limit.
Thus $\lim_{t\to 1}G_t \not =0$ and the two measures are equivalent.
\end{proof}

\def\cprime{$'$} \def\cprime{$'$}

\end{document}